\documentclass [draft,envcountsect,envcountsame] {svjour3}

\title{An Erd\H{o}s--Gallai-type theorem for keyrings}

\author{Alexander Sidorenko}
\institute{
A. Sidorenko 
ORCID: 0000-0002-1755-4013
\at 16 Sunrise Drive, Armonk, NY 10504, USA 
\\\email sidorenko.ny@gmail.com
}

\date{\today}

\begin{document}
\maketitle

\begin{abstract}
A keyring is a graph obtained by appending $r \geq 1$ leaves 
to one of the vertices of a cycle. 
We prove that for every $r \leq (k-1)/2$, 
any graph with average degree more than $k-1$ 
contains a keyring with $r$ leaves and at least $k$ edges. 
\end{abstract}


\section{Introduction}

The graphs considered in this paper are simple and undirected. 
The sets of vertices and edges of a graph $G$ 
are denoted by $V(G)$ and $E(G)$, respectively. 
Their sizes are $v(G)=|V(G)|$ and $e(G)=|E(G)|$. 
For $u \in V(G)$, 
we denote by $N_G(u)$ the set of vertices adjacent to $u$ in $G$, 
and $d_G(u)=|N_G(u)|$ is the {\it degree} of $u$. 
For a subset of vertices $X \subset V(G)$, 
the number of edges of $G$ with at least one end in $X$ 
is denoted by $d_G(X)$.

Let $\mathcal{D}_k$ be the class of all graphs 
whose average degree is strictly more than $k-1$. 
In other words, $G \in\mathcal{D}_k\;$ if $\;2e(G) > (k-1) v(G)$. 
Erd\H{o}s and Gallai proved the following two statements.

\begin{theorem}[Theorem 2.6 \cite{EG}]
\label{EG_path}
Any graph $G \in \mathcal{D}_k$ contains a path of $k$ edges.
\end{theorem}

\begin{theorem}[Theorem 2.7 \cite{EG}]
\label{EG_cycle}
Any graph $G \in \mathcal{D}_k$ contains a cycle with at least $k$ edges.
\end{theorem}

In fact, the statement of Theorem~2.7 in \cite{EG} is even stronger: 
an $n$-vertex graph without cycles of length $k$ (or more) 
has at most $(n-1)(k-1)/2$ edges, 
and less than that if $n-1$ is not a multiple of $k-2$. 
Faudree and Schelp~\cite{Faudree} 
and independently Kopylov~\cite{Kopylov} 
determined (for every $n$ and $k$) the largest number of edges 
in an $n$-vertex graph without a $k$-edge path. 
They also described the extremal graphs. 
Kopylov~\cite{Kopylov} did the same for 2-connected graphs 
without cycles of length $k$ and more.

Together, Theorem \ref{EG_path} and a simple observation 
that every graph $G \in \mathcal{D}_k$ contains a star with $k$ edges, 
led to the following conjecture formulated by Erd\H{o}s and S\'{o}s (see \cite{Erd}).

\vspace{0.3cm}

\noindent
{\bf Erd\H{o}s--S\'{o}s Conjecture.}
{\it
Any graph $G \in \mathcal{D}_k$ contains every tree with $k$ edges.
}

\vspace{0.3cm}

Let $\mathcal{T}_k$ be the class of $k$-edge trees $T$ such that 
every $G \in \mathcal{D}_k$ contains $T$ as a subgraph.
The Erd\H{o}s--S\'{o}s conjecture states that 
all $k$-edge trees belong to $\mathcal{T}_k$.
Ajtai, Koml\'{o}s, Simonovits, and Szemer\'{e}di \cite{AKSS1,AKSS2,AKSS3} 
proved that there exists $k_0$ such that 
the conjecture holds for all $k > k_0$.
Still, the general case has not been solved, 
and only partial results have been obtained.

A {\em spider} $S(k_1,k_2,\ldots,k_r)$ is a tree obtained 
from $r$ disjoint paths of lengths $k_1,k_2,\ldots,k_r$ 
by combining their starting vertices into one. 
The combined vertex has degree $r$ and is called the {\it center}. 
Obviously, $S(k_1,k_2)$ is just a path of length $k_1+k_2$. 
Wo\'{z}niak \cite{Woz} proved that 
$S(k_1,k_2,\ldots,k_r) \in \mathcal{T}_k$ if $k_1,k_2,\ldots,k_r \leq 2$.
Fan and Sun \cite{FS} proved that
$S(k_1,k_2,\ldots,k_r) \in \mathcal{T}_k$ 
when $r=3$ or $k_1,k_2,\ldots,k_r \leq 4$. 
Very recently, Fan, Hong and Liu~\cite{FHL}
proved that all spiders belong to $\mathcal{T}_k$. 
McLennan \cite{McL} proved that $T \in \mathcal{T}_k$ 
when $T$ is a tree of diameter 4. 

It was mentioned in Section~3 of~\cite{Mos} that Perles 
proved the Erd\H{o}s--S\'{o}s conjecture for caterpillars 
(those are trees which do not contain $S(2,2,2)$ as a subgraph). 
The proof was published in~\cite{Kalai} only recently. 

A vertex which is adjacent to a leaf is called a {\it preleaf}.
It was proved in~\cite{Sid} that 
if a tree $T$ has a preleaf 
which is adjacent to at least $(k-1)/2$ leaves 
then $T \in \mathcal{T}_k$. 
We will prove in Section~\ref{proof_preleaves} a stronger statement:

\begin{theorem}
\label{preleaves}
If a tree $T$ with $k$ edges has $p$ preleaves 
and one of them is adjacent to at least $(k-p-1)/2$ leaves 
then $T \in \mathcal{T}_k$.
\end{theorem}

Trees and cycles are not the only subgraphs whose existence can be deduced 
from the graph's average degree. 
As a referee of this paper pointed out, it was Tur\'{a}n who formulated 
the very problems for which the Erd\H{o}s--Gallai theorems provide the answers. 
He asked the maximal number of edges in an $n$-vertex graph 
which does not contain a {\it lasso}, 
that is a cycle and a path having one common vertex. 
Fan and Sun \cite{FS} solved the lasso problem 
(but did not formulate their result explicitly) 
within the proof of their Theorem~3.1. 
In this paper, we consider a similar forbidden pattern. 
A {\it keyring} $C_r(l)$ is a $(l+r)$-edge graph 
obtained from a cycle of length $l$ 
by appending $r$ leaves to one of its vertices. 
This vertex has degree $r+2$ and is called the {\it center} of the keyring.
In Section~\ref{proof_keyring}, 
we prove an analog of Theorem \ref{EG_cycle} for keyrings:

\begin{theorem}
\label{keyring}
For any positive integer $r \leq (k-1)/2$, 
every graph $G \in \mathcal{D}_k$ 
contains a keyring with $r$ leaves and at least $k$ edges.
\end{theorem}

A graph is called $k$-{\it minimal} if it belongs to $\mathcal{D}_k$ 
but none of its proper subgraphs belongs to  $\mathcal{D}_k$. 
Obviously, any graph from $\mathcal{D}_k$ contains a $k$-minimal subgraph. 
Thus, in proofs of Theorems \ref{keyring}, \ref{preleaves} 
and similar statements, 
instead of considering all graphs $G \in \mathcal{D}_k$, 
it is sufficient to consider only those which are $k$-minimal. 
The main help comes from the simple observation:

\begin{remark}
\label{minimal}
If a graph $G$ is $k$-minimal then for any subset of vertices $X \subset V(G)$ 
the number of edges of $G$ with at least one end in $X$ 
is strictly more than $\frac{k-1}{2} |X|$.
\end{remark}

\section{Proof of Theorem \ref{preleaves}}
\label{proof_preleaves}

Let $m(T)$ denote the largest number of leaves connected to a single preleaf in $T$, 
$L(T)$ denote the set of leaves of $T$,
and $P(T)$ denote the set of preleaves.
In this proof, we will keep $k$ fixed and use induction in $m(T)$. 
The basis of induction is the case $m(T) = k$. 
In this case, $T$ is a star and belongs to $\mathcal{T}_k$. 
Now we are going to prove the inductive step. 
Suppose that $m(T) = m < k$ 
and the statement of the theorem holds for all trees $T'$ where $m(T') > m$. 
(We will make use of the assumption $m \geq (k-p-1)/2$ later.) 

Consider a $k$-minimal graph $G$. 
We need to show that $G$ contains $T$ as a subgraph.
Let $u$ be a preleaf of $T$ with $m$ leaves. 
Let $v$ be another preleaf of $T$ (since $m < k$, 
$\;T$ is not a star and has at least two preleaves). 
Now we disconnect in $T$ one of the leaves attached to $v$ and reconnect it to $u$ instead. 
The resulting tree $T'$ has $k$ edges and $m(T') = m+1$. 
Since $|P(T')| \geq |P(T)|-1$, 
we have $m(T') = m+1 \geq (k - |P(T')| - 1) / 2$, 
and by the induction hypothesis, $G$ must contain a copy of $T'$. 
We are going to transform this copy of $T'$ into a copy of $T$ 
by changing the assignment of leaves to the preleaves. 
From now on, we will assume that $V(T')$ is a subset of $V(G)$, 
and $E(T')$ is a subset of $E(G)$. 
We define three non-overlapping sets of vertices: 
$A = P(T') \cup \{ v \}$, 
$B = L(T') \; \backslash  \; \{ v \}$, 
$\; C = V(G) \; \backslash \; V(T')$. 
Notice that $v$ in $T'$ might remain a preleaf or become a leaf or neither.
In any case, we want to ensure that $v$ belongs to $A$ and not to $B$.
Thus, $|A| = |P(T)| = p$, $|B| = |L(T)|$.
Consider a bipartite directed graph $F$ whose vertex set is $A \cup B \cup C$ 
with directed edges of two types: 
\begin{enumerate}
  \item $(a,b)$ where $a \in A$, $\;b \in (B \cup C)$ 
        and $\{ a,b \} \in E(G)$; and 
  \item $(b,a)$ where $a \in A$, $\;b \in B$, 
        $\;a$ and $b$ are adjacent in $T'$ 
        (which means that leaf $a$ is connected to preleaf $b$).
\end{enumerate}
If there exists a path in $F$ from $v$ either to $u$ or to $C$, 
we will be able to find a copy of $T$ in $G$. 
Indeed, let $(a_0,b_0,a_1,b_1,\ldots,a_q, b_q)$ be a simple path 
where $a_0 = v$, $b_q \in C$, 
$\{ a_i, b_i \} \in E(G)$ $(i=0,1,\ldots,q)$, 
and $b_i \in B$, $\;a_{i+1} \in A$ 
are adjacent in $T'$ $(i=0,1,\ldots,q-1)$. 
Then we can add to $T'$ vertex $b_q$ 
as well as edges $\{ a_i, b_i \}$ for $i=0,1,\ldots,q$, 
remove edges $\{ a_{i+1}, b_i \}$ for $i=0,1,\ldots,q-1$ 
and remove one of the leaves connected to $u$. 
The resulting subgraph of $G$ is a copy of $T$. 

Similarly, let $(a_0,b_0,a_1,b_1,\ldots,a_q)$ be a simple path 
where $a_0 = v$, $a_q = u$, 
$\{ a_i, b_i \} \in E(G)$ $(i=0,1,\ldots,q-1)$, 
and $b_i \in B$, $\;a_{i+1} \in A$ 
are adjacent in $T'$ $(i=0,1,\ldots,q-1)$. 
Then we can add to $T'$ edges $\{ a_i, b_i \}$ for $i=0,1,\ldots,q-1$ 
and remove edges $\{ a_{i+1}, b_i \}$ for $i=0,1,\ldots,q-1$. 
The resulting subgraph of $G$ is a copy of $T$. 

Finally, consider the case when 
$C \cup \{ u \}$ is unreachable from $v$ in $F$. 
We split A into two subsets: 
$X$ consists of the vertices that are reachable from $v$ in $F$, 
and $Y$ consists of the rest. 
Obviously, $v \in X$ and $u \in Y$. 
Let $Z$ be the set of $m+1$ leaves that are attached to $u$.
There are no edges in $G$ between $X$ and $C \cup Y \cup Z$. 
Since $|X|+|Y|=p$ and $|Z| = m+1 \geq (k-p+1)/2$, 
then for any $w \in X$ we can estimate: 
\begin{eqnarray*}
 |N_G(w) \backslash X| \; & \leq & \; |V(T')|-(|X|+|Y|+|Z|) \\
                       \; & \leq & \; (k+1) - \left( p + \frac{k-p+1}{2} \right)
                       \;\;\; = \;\;\; \frac{k-p+1}{2} \; .
\end{eqnarray*}
Using $|X| \leq p-1$, we get
\begin{eqnarray*}
 d_G(X) \; & \leq & \; \frac{|X|(|X|-1)}{2} \;\;\; + \;\;\; |X| \; \frac{k-p+1}{2} 
        \;\;\; = \;\;\; |X| \; \frac{|X|+k-p}{2} \\
        \; & \leq & \; \frac{(p-1)+k-p}{2} \; |X| \;\;\; = \;\;\; \frac{k-1}{2} \; |X| 
\end{eqnarray*}
which, according to Remark \ref{minimal}, 
contradicts the $k$-minimality of $G$.
\qed

\section{Proof of Theorem \ref{keyring}}
\label{proof_keyring}

\begin{lemma}
\label{Hamiltonian}
Fix integers $\lambda \geq 2$ and $t \geq 1$. 
Let $H$ be a Hamiltonian graph with $m \geq \lambda$ vertices, 
and $u_0$ be one of them.
If $d_H(u_0) \geq 2t-1 + \max \{ 2\lambda-m-1, 2 \}$ 
then there exists $l \geq \lambda$ 
such that $H$ contains a copy of $C_t(l)$ with the center at $u_0$.
\end{lemma}

\begin{proof}
Let $(u_0,u_1,\ldots,u_{m-1},u_m=u_0)$ be a Hamiltonian cycle in $H$. 
Suppose first that $m=\lambda$. Since $m \geq 3$, 
we get $2\lambda-m-1 = m-1 \geq 2$ 
and $d_H(u_0) \geq 2t+m-2 \geq m = v(H)$ which is impossible.
Therefore, we may assume $m \geq \lambda+1$.
Denote 
\[ X_1 = \{u_2,u_3,\ldots,u_{m-\lambda+1}\}, \]
\[ X_2 = \{u_{\lambda-1},u_{\lambda},\ldots,u_{m-2}\}, \]
\[ X_3 = \{u_1,u_2,\ldots,u_{m-1}\} \; \backslash \; (X_1 \cup X_2). \]
Clearly, $|X_1 \cup X_2| = |X_1| + |X_2| = 2(m-\lambda)$ 
when $m-\lambda+1 < \lambda-1$, 
and $|X_1 \cup X_2| = m-3$ when $m-\lambda+1 \geq \lambda-1$.
Thus, $|X_1 \cup X_2| = \min \{ 2m-2\lambda, m-3 \}$ 
and $|X_3| = (m-1) - |X_1 \cup X_2| = \max \{ 2\lambda-m-1, 2 \}$.
Note that
\begin{eqnarray*}
 |N_H(u_0) \cap (X_1 \cup X_2)| & \geq & d_H(u_0) - |X_3| \\
    & \geq & (2t-1) + \max \{ 2\lambda-m-1, 2 \} - |X_3| \;\; = \;\; 2t-1.
\end{eqnarray*}
Thus, there exist 
$u_{i_1},u_{i_2},\ldots,u_{i_{2t-1}} \in N_H(u_0) \cap (X_1 \cup X_2)$ 
where $i_1 < i_2 < \ldots < i_{2t-1}$.
If $u_{i_t} \in X_1$ then $i_t \leq m-\lambda+1$.
In this case, $u_0$ is adjacent to 
$u_1,u_{i_1},u_{i_2},\ldots,\allowbreak u_{i_{t-1}}$ 
and belongs to the cycle $(u_0,u_{i_t},u_{i_t+1},\ldots,u_m=u_0)$ 
whose length is $m - i_t + 1$. 
This produces a copy of $C_t(l)$ with the center at $u_0$ 
where $l = m - i_t + 1 \geq m - (m-\lambda+1) + 1 = \lambda$.
Alternatively, if $u_{i_t} \in X_2$ then $i_t \geq \lambda-1$. 
In this case, $u_0$ is adjacent to $u_{i_{t+1}},u_{i_{t+2}},\ldots,u_{i_{2t-1}},u_{m-1}$ 
and belongs to the cycle $(u_0,u_1,\ldots,u_{i_t},u_m=u_0)$ 
whose length is $i_t + 1$. 
This produces a copy of $C_t(l)$ with the center at $u_0$ 
where $l = i_t + 1 \geq \lambda$.
\qed
\end{proof}

\noindent
{\it Proof of Theorem \ref{keyring}}
We are going to show that any $k$-minimal graph $G$ contains a copy of $C_r(l)$.
By Theorem \ref{EG_cycle}, $G$ contains a cycle 
$(u_0,u_1,\ldots,u_{m-1},u_m=u_0)$ of length $m \geq k$. 
Let $X = \{ u_0,u_1,\ldots,u_{m-1} \} $. 
Then $d_G(u_i) =  t_i + r_i$ where 
$t_i = |N_G(u_i) \cap X|$ and $r_i = |N_G(u_i) \backslash X|$ 
for $i=0,1,\ldots,m-1$. 
According to Remark \ref{minimal}, 
\[ \sum_{i=0}^{m-1} (t_i + 2r_i) = 2d_G(X) > m \cdot (k-1) . \]
Thus, there is such an index $i$ that $t_i + 2r_i \geq k$. 
If $r_i \geq r$ then $G$ contains a copy of $C_r(m)$ with the center at $u_i$. 
Suppose $r_i < r$. Let $H$ be the subgraph of $G$ induced by $X$.
We are going to apply Lemma \ref{Hamiltonian} with parameters 
$t = r - r_i$ and $\lambda = k-r+1$.
To be able to invoke it, we need to demonstrate that 
$\lambda \geq 2$ and 
\begin{eqnarray*}
  d_H(u_i) = t_i & \geq & 2(r-r_i) - 1 + \max \{ 2(k-r+1) - m - 1, 2 \} \\
   & = & \max \{ 2k - 2r_i - m, 2r - 2r_i + 1 \} \; .
\end{eqnarray*}
Indeed, on one hand, since $t_i + 2r_i \geq k$, we get 
$t_i \geq k - 2r_i = 2k - 2r_i - k \geq 2k - 2r_i - m$. 
On the other hand, since $r \leq (k-1)/2$, we get 
$t_i \geq k - 2r_i \geq 2r - 2r_i + 1$. 
Also, $r \leq (k-1)/2$ implies $\lambda = k-r+1 \geq 2$.
By Lemma \ref{Hamiltonian}, 
$H$ contains a copy of $C_t(l)$ with the center at $u_i$ 
where $t=r-r_i$ and $l \geq \lambda = k-r+1$. 
Now we use the $r_i$ vertices from $N_G(u_i) \backslash X$ to extend it to a copy of $C_r(l)$ 
where $l+r \geq k+1$.
\qed

\begin{acknowledgements}
The author would like to thank the referees for their suggestions and comments.
\end{acknowledgements}

\end{document}